\documentclass[12pt]{amsart}
\usepackage{a4}
\usepackage{latexsym}
\usepackage{amsfonts}
\usepackage{amssymb}
\usepackage{amsmath}
\usepackage{amsthm}
\newtheorem{theorem}{Theorem}[section]
\newtheorem{lemma}[theorem]{Lemma}

\newtheorem{application}[theorem]{Application}
\title{An improved sum-product estimate for general finite fields}
\author{Liangpan Li}
\address{Department of Mathematics\\ Shanghai Jiao Tong University,
Shanghai 200240\\  China 
\& Department of Mathematical Sciences,
Loughborough University, LE11 3TU, UK}
\email{liliangpan@gmail.com}
\author{Oliver Roche-Newton}
\address{Department of Mathematics\\ University Walk\\ Bristol\\ BS8 1TW\\ UK}
\email{maorn@bristol.ac.uk}

\date{}

\begin{document}

\begin{abstract}

This paper gives an improved sum-product estimate for subsets of a finite field whose order is not prime. It is shown, under certain conditions, that
$$\max\{|A+A|,|A\cdot{A}|\}\gg{\frac{|A|^{12/11}}{(\log_2|A|)^{5/11}}}.$$
This new estimate matches, up to a logarithmic factor, the current best known bound obtained over prime fields by Rudnev (\cite{mishaSP}).

\end{abstract}

\maketitle

\section{Introduction}

Let $A$ be a subset of $F=\mathbb{F}_{p^n}$ for some prime $p$ and
some $n\in{\mathbb{N}}$. Consider the sumset and productset of $A$,
defined respectively as
\begin{align*}A+A&:=\{a+b:a,b\in{A}\},\\
A\cdot{A}&:=\{ab:a,b\in{A}\}.\end{align*} An interesting problem is
to establish lower bounds on the quantity
$\max{\{|A+A|,|A\cdot{A}|\}}$. The existence of non-trivial bounds
was first established by Bourgain, Katz and Tao (\cite{BKT}) for the
case of prime fields. Garaev (\cite{garaev}) established the first
quantitative sum-product estimate for fields of prime order.
Garaev's result can be stated as follows:

\begin{theorem}\label{theorem:garaev}
Let $A\subset{\mathbb{F}_p}$ for some prime $p$, such that $|A|\leq{p^{7/13}}(\log{p})^{-4/13}$. Then
\[\max\{|A+A|,|A\cdot{A}|\}\gg{\frac{|A|^{15/14}}{(\log{|A|})^{2/7}}}.\]

\end{theorem}

This result was generalised, for finite fields whose order is not
necessarily prime, by Katz and Shen (\cite{KS}) in the form of the
following theorem:

\begin{theorem}\label{theorem:KStheorem}
Let $F=\mathbb{F}_{p^n}$ be a finite field. Suppose that $A$ is a subset of $F$ so that
for any subfield $G\subseteq{F}$, and any elements $c,d\in{F}$,
\[|A\cap{(cG+d)}|\leq{\max\{|G|^{1/2},|A|^{18/19}\}}.\]
Then it must be the case that
\[\max{\{|A+A|,|A\cdot{A}|\}}\gg{\frac{|A|^{20/19}}{(\log{|A|})^{\alpha}}},\]
where $0<{\alpha}\leq{1}$ is some absolute constant.

\end{theorem}

Since Garaev's sum-product estimate (\cite{garaev}), there have been a number of small improvements made courtesy of more subtle arguments. First, Katz and Shen (\cite{KSprime}) increased the power of $|A|$ to $\frac{14}{13}$. Bourgain and Garaev (\cite{BG}) and Shen (\cite{ShenSP}) then improved this to $\frac{13}{12}$, and the first author (\cite{Li}) found a way of removing the logarithmic factor from the estimate. Recently, Rudnev (\cite{mishaSP}) introduced a new technique for the question over prime fields which gives a further improvement; so the following result of Rudnev represents the current state of the art.

\begin{theorem}\label{theorem:mishaSP}
Let $A\subset{\mathbb{F}_p^*}$ with $|A|<p^{1/2}$ and $p$ large. Then
\[\max{\{|A+A|,|A\cdot{A}|\}}\gg{\frac{|A|^{12/11}}{(\log{|A|})^{4/11}}}.\]
\end{theorem}

Utilising some of the techniques from these papers, Shen (\cite{shenSPnonprime}) improved the power of $|A|$ in Theorem \ref{theorem:KStheorem} to $\frac{17}{16}$, and in a recent preprint, the second author (\cite{ORN}) increased this to $\frac{15}{14}$ under similar conditions.

This paper seeks to improve further upon Theorem \ref{theorem:KStheorem}, with the argument differing in a few places. Most importantly, the structure of the proof is rearranged so that the worst case is less damaging to the final estimate. This restructuring has the additional benefit of making the conditions of Theorem \ref{theorem:main} less restrictive and more natural. Also, the new idea of Rudnev (\cite{mishaSP}), which gave the improvement in Theorem \ref{theorem:mishaSP}, is incorporated. The outcome is the following result:

\begin{theorem}\label{theorem:main}
Let $A$ be a subset of $\mathbb{F}_{p^n}^*$. If
$|A\cap{cG}|\leq{|G|^{1/2}}$ for  any subfield $G$ of
$\mathbb{F}_{p^n}$ and any element $c\in{F}$, then
\[\max\{|A+A|,|A\cdot{A}|\}\gg{\frac{|A|^{12/11}}{(\log_2|A|)^{5/11}}}.\]

\end{theorem}

\subsection{Some remarks concerning Theorem \ref{theorem:main}}

The statement of Theorem \ref{theorem:main} is in fact quite flexible.
 The tools used do not distinguish between addition and subtraction, which means that the difference
 set, $A-A$, can replace $A+A$ in the above. It is also possible to get a sum-ratio estimate,
 where $A\cdot{A}$ is replaced by the ratio set $A:A$, which is defined as follows:
$$A:A\stackrel{\mathrm{def}}{=}\left\{\frac{a}{b}\text{ such that }a,b\in{A},b\neq{0}\right\}.$$
The presentation of a sum-ratio estimate proof is of a very similar nature, but rather more simple,
as Rudnev (\cite{mishaSP}) alluded to. Lemma \ref{theorem:misha} can be replaced by a more straightforward
pigeonholing argument and there is no need for dyadic pigeonholing, which means less technicalities are required
and no logarithmic factor appears in the final estimate. The proof still requires a bit of work, but has been
omitted from this paper because it is too similar to the proof of Theorem \ref{theorem:main}.

\subsection{Notation}

Throughout this paper, the symbols $\ll,\gg$ and $\approx$ are used to suppress constants. For example, $X\ll{Y}$ means that there exists some absolute constant $C$ such that $X<CY$. $X\approx{Y}$ means that $X\ll{Y}$ and $Y\ll{X}$. When describing such a rough inequality in general language, inverted commas are used in an effort to avoid confusion. For example, the statement that ``$X$ is `at most' $Y$'' tells us that $X\ll{Y}$. A similar meaning is attached to `at least'.

\subsection{Acknowledgements}

The first author would like to thank Chun-Yen Shen for helpful
discussions. He was supported by the Natural Science Foundation of
China (Grant Number 11001174). The second author would like to thank
Timothy Jones and Misha Rudnev for some extremely helpful
conversations and typographical corrections.

\section{Preliminary results}

Before proving the main theorem, it is necessary to state some previous results. The first two results are well-known within arithmetic combinatorics, and have been crucial to all known quantitative sum-product estimates over finite fields. The first is due to Pl\"{u}nnecke and Rusza (\cite{ruzsa}), whilst the second is a generalisation which Katz and Shen (\cite{KSprime}) used to obtain an improvement on the original quantitative sum-product estimate of Garaev (\cite{garaev}).

\begin{lemma}\label{theorem:SPlun}Let $X,B_1,...,B_k$ be subsets of $F$. Then
$$|B_1+\cdots+B_k|\leq{\frac{|X+B_1|\cdots|X+B_k|}{|X|^{k-1}}}.$$

\end{lemma}

\begin{lemma}\label{theorem:CPlun}
Let $X,B_1,...,B_k$ be subsets of $F$. Then for any $0<\epsilon<1$,
there exists a subset $X'\subseteq{X}$, with
$|X'|\geq{(1-\epsilon)|X|}$, and some constant $C(\epsilon)$, such
that
$$|X'+B_1+\cdots+B_k|\leq{C(\epsilon)\frac{|X+B_1|\cdots|X+B_k|}{|X|^{k-1}}}.$$

\end{lemma}

We will need the following covering lemma, which appeared in sum-product estimates for the first time in Shen (\cite{ShenSP}). An application of this lemma has been the key to the two of the most recent improvements to the sum-product estimate over prime fields (see Rudnev (\cite{mishaSP}) and Shen (\cite{ShenSP})).

\begin{lemma}\label{theorem:covering}
Let $X$ and $Y$ be additive sets. Then for any $\epsilon\in{(0,1)}$ there is some constant $C(\epsilon)$, such that at least $(1-\epsilon)|X|$ of the elements of $X$ can be covered by $C(\epsilon)\frac{\min\{|X+Y|,|X-Y|\}}{|Y|}$ translates of $Y$.

\end{lemma}

The next result has been extracted from case 2 in the proof of the
main theorem in Rudnev (\cite{mishaSP}). A proof is presented here
for completeness. For any set $B\subset{F}$, define $R(B)$ to be the
set
\begin{equation}
\label{Rdefn}
R(B):=\left\{\frac{b_1-b_2}{b_3-b_4}:b_1,b_2,b_3,b_4\in{B},b_3\neq{b_4}\right\}.
\end{equation}

\begin{lemma}\label{lemma 21}
  Let $B\subset F$ with $|R(B)|\gg|B|^2$.
  Then there exist elements $a,b,c,d\in B$ such that for any subset $B'\subset B$
  with $|B'|\approx|B|$,
 $$|(a-b)\cdot B'+(c-d)\cdot B'|\gg|B|^2.$$

\end{lemma}

\begin{proof}
Let $E^{\oplus}(X,Y)$ be the additive energy of nonempty subsets
$X,Y\subset F$, that is,
$$E^{\oplus}(X,Y)=\left|\{(x_1,y_1,x_2,y_2)\in X\times Y\times X\times Y: x_1+y_1=x_2+y_2\}\right|.$$
From the Cauchy-Schwarz inequality it follows that
$E^{\oplus}(X,Y)\geq\frac{|X|^2|Y|^2}{|X+ Y|}$. Note that
\begin{align*}
\sum_{r\in R(B)}E^{\oplus}(B,rB)&=\left|\{(x_1,y_1,x_2,y_2,r)\in
B^4\times R(B):b_1+rb_2=b_3+rb_4\}\right|
\\&\leq|B|^2\cdot|R(B)|+|B|^4\ll{|B|^2\cdot|R(B)|}.
\end{align*}
Hence, there exists an element $\hat{r}\in R(B)$ with
$\hat{r}=\frac{a-b}{c-d}$ such that
$E^{\oplus}(B,\hat{r}B)\ll{|B|^2}$. Now for any subset $B'\subset B$
with $|B'|\approx|B|$,
$$|B'+\hat{r}B'|\geq\frac{|B'|^4}{E^{\oplus}(B',\hat{r}B')}\gg\frac{|B|^4}{E^{\oplus}(B,\hat{r}B)}\gg|B|^2.$$
This proves the desired inequality.

\end{proof}

\begin{lemma}\label{lemma 22} Let  $B$ be a nonempty subset of
$F^*$ and $\mathbb{F}_B$ be the subfield generated by $B$. Then for
any element $z\in \mathbb{F}_B$, there exists a polynomial of
several variables with integer coefficients $P(x_1,x_2,\ldots,x_m)$
and a sequence of elements $(b_1,b_2,\ldots,b_m)\in B^m $ such that
$P(b_1,b_2,\ldots,b_m)=z$.
\end{lemma}

\begin{proof}
Let $Z$ be the set of elements $z\in \mathbb{F}_B$ satisfying the
designated property in this lemma.  Let $z_1 z_2$ be any two
elements of $Z$. Obviously, $z_1+z_2,z_1z_2\in Z$. We also note
$1\in Z$ (due to the fact that $(|F|-1)$-fold multiplication of any
non-zero element equals one) and $0\in Z$ (due to the fact that
$(|\mbox{char}( F)|)$-fold addition of any non-zero element equals
zero).
 This naturally implies that $Z$ is a subfield
of $F$. We can easily observe that $B$ is a subset of $Z$, and
consequently, $\mathbb{F}_B\subset Z$. But considering $Z$ is a
subset of $\mathbb{F}_B$, we are done. This concludes the proof.

\end{proof}

\begin{lemma}\label{theorem:polynomial}
Let $B$ be a subset of $F$ with at least two elements, and
$\mathbb{F}_{B}$ be the subfield generated by $B$. Then there exists
a polynomial of several variables with integer coefficients
$P(x_1,x_2,\ldots,x_{m})$ such that
$P(B,B,\ldots,B)=\mathbb{F}_{B}.$
\end{lemma}

\begin{proof}
By Lemma \ref{lemma 22}, for any element $z\in \mathbb{F}_B$, there
exists a polynomial $P_z$ of several variables with positive integer
coefficients $P_z(x_1^{(z)},x_2^{(z)},\ldots,x_{m_z}^{(z)})$ and a
sequence of elements $(b_1,b_2,\ldots,b_{m_z})\in B^{m_z}$ such that
$P_z(b_1,b_2,\ldots,b_{m_z})=z$. Let
\[Q_{z}(x_1^{(z)},x_2^{(z)},\ldots,x_{m_z+2}^{(z)}):=(x_{m_z+2}^{(z)}-x_{m_z+1}^{(z)})\cdot P_z(x_1^{(z)},x_2^{(z)},\ldots,x_{m_z}^{(z)})\]
and
\[P(\{x^{(z)}_i: z\in \mathbb{F}_B, 1\leq i\leq m_z+2\}):=\sum_{z\in \mathbb{F}_B}Q_z(x_1^{(z)},x_2^{(z)},\ldots,x_{m_z+2}^{(z)}).\]
Let $c,d\in B$ be two different elements. Then
$$\mathbb{F}_B=\sum_{z\in \mathbb{F}_B}\{0,(d-c)\cdot z\}\subset\sum_{z\in \mathbb{F}_B}Q_z(\underbrace{B,B,\ldots,B}_{m_z+2})
=P(\underbrace{B,B,\ldots,B}_{s})\subset\mathbb{F}_{B},$$ 
where
$s:=\sum_{z\in \mathbb{F}_B}(m_z+2)$. This finishes the
proof.
\end{proof}

The rest of this paper is devoted to proving Theorem \ref{theorem:main}.

\section{Proof of Theorem \ref{theorem:main}}

Let $A$ be a set satisfying the conditions of Theorem \ref{theorem:main}, and suppose that $|A+A|,|A\cdot{A}|\leq{K|A|}$.
 The aim is to show that $K\gg{\frac{|A|^{1/11}}{(\log{|A|})^{5/11}}}$.

At the outset, apply Lemma \ref{theorem:CPlun} to identify some
subset $A'\subset{A}$, with cardinality $|A'|\approx{|A|}$, so that
$$|A'+A'+A'+A'|\ll{\frac{|A+A|^3}{|A|^2}}\ll{K^3|A|}.$$
Since many more refinements of $A$ are needed throughout the proof,
this first change is made without a change in notation. So,
throughout the rest of the proof, when the set $A$ is referred to,
we are really talking about the large subset $A'$, which satisfies
the above inequality. In other words, we assume that
\begin{equation}
|A+A+A+A|\ll{K^3|A|}.
\label{finaltouch}
\end{equation}
Consider the point set $A\times{A}\subset{F\times{F}}$. The
multiplicative energy, $E(A)$, of $A$ is defined to be the number of
solutions to
\begin{equation}
\frac{a_1}{a_2}=\frac{a_3}{a_4},
\label{menergy}
\end{equation}
such that $a_1,a_2,a_3,a_4\in{A}$. Let $\mathcal{L}$ be the set of
all lines through the origin. By definition,
$$E(A)=\sum_{l\in{\mathcal{L}}}|l\cap({A\times{A}})|^2.$$
By pigeonholing, we can identify a popular dyadic group.
This is done by partitioning lines through the origin according to
their popularity. Speaking more precisely, let $\mathcal{L}_i$ be
 the set of all lines $l$ through the origin such that $2^i\leq{|l\cap{A\times{A}}|}<{|2^{i+1}|}$.
  Therefore the multiplicative energy may be rewritten in the form
$$E(A)=\sum_{j=0}^{\lfloor\log_2|A|\rfloor}\sum_{l\in{\mathcal{L}_j}}|l\cap({A\times{A}})|^2.$$
Then, by elementary piegeonholing, one may choose a particular
dyadic group, which contributes more than the average for this sum.
Therefore, there exists a set of $L$ lines, each supporting
$\approx{N}$ points from $A\times{A}$ such that
\begin{equation}
M:=LN^2\gg{\frac{E(A)}{\log_2|A|}}\geq{\frac{|A|^4}{|A\cdot{A}|\log_2|A|}}\geq{\frac{|A|^3}{K\log_2|A|}},
\label{dyadic}
\end{equation}
where the second inequality comes from the standard Cauchy-Schwarz
bound on the multiplicative energy (see (\cite{TV})). Refine the point set by now
considering only points that lie on these $L$ lines, and call this
set $P\subset{A\times{A}}$. Clearly, $|P|\approx{LN}$. Denote by
$\Xi$ the set of slopes through the origin in this refined set $P$.
More precisely, define $\Xi$ as follows:
$$\Xi:=\left\{\frac{b}{a}:(a,b)\in{P}\right\}.$$
The fact that $LN\leq{|A|^2}$ easily implies that
\begin{equation}
N\geq{\frac{M}{|A|^2}}.
\label{Nbound}
\end{equation}
Similarly, since $N\leq{|A|}$, it is clear that
\begin{equation}
L\geq{\frac{M}{|A|^2}}.
\label{Lbound}
\end{equation}
Let $A_x$ be the set of ordinates of $P$ for a fixed abscissa $x$,
and $B_y$ be the set of abscissae of $P$ for a fixed ordinate $y$,
that is,
\begin{align*}
A_x&=\{y:(x,y)\in{P}\}\\
B_y&=\{x:(x,y)\in{P}\}.\end{align*} A little more notation is
required still. For some element $\xi\in{\Xi}$, let $P_{\xi}$ be the
projection of points in $P$ on the line with equation $y=\xi{x}$
onto the x-axis. So,
$$P_{\xi}=\{x:(x,\xi{x})\in{P}\},$$
and since all lines with slope in $\Xi$ intersect $P$ with cardinality approximately $N$,
it follows that $|P_{\xi}|\approx{N}$ for all $\xi{\in{\Xi}}.$
Note also, since $P\subset{A\times{A}}$, that $P_{\xi}$ is a subset of $A$.

The following lemma is taken from Rudnev (\cite{mishaSP}).

\begin{lemma}\label{theorem:misha}
 There exists a popular abscissa $x_0$ and a
 popular ordinate $y_0$ (popular here means that $|A_{x_0}|,|B_{y_0}|\gg{\frac{LN}{|A|}}$),
 as well as a large subset $\tilde{A}_{x_0}\subseteq{A_{x_0}}$, with
\begin{equation}
|\tilde{A}_{x_0}|\gg{\frac{LM}{|A|^3}},
\label{mishalemma1}
\end{equation}
such that for every $z\in{\tilde{A}_{x_0}}$,
\begin{equation}
|\tilde{A}_{\tilde{z}}:=P_{z/x_{0}}\cap{B_{y_0}}|\gg{\frac{LMN}{|A|^4}}.
\label{mishalemma2}
\end{equation}

\end{lemma}

Since the sum-product estimate and the conditions of Theorem \ref{theorem:main} are invariant under dilation, we may
assume without loss of generality that $x_0=1$. This means that
elements of $A_{x_0}$ are also the popular slopes described above,
i.e.
$$A_{x_0}\subset{\Xi}.$$

\subsection{Application of the covering lemma}

Since Lemma \ref{theorem:covering} is applied in a similar manner
several times throughout the remainder of the proof, it is worthwhile
highlighting this in advance, in order to to avoid repetition.

\begin{application}\label{theorem:application}
If $A'\subset{A}$ and $\xi\in{\Xi}$, then $\pm{\xi{A'}}$ can be 90\% covered by `at most' $\frac{K|A|}{N}$ translates of $A$.

\end{application}

\begin{proof} 90\% of $\xi{A'}$ can be covered by `at most'
$$\frac{|\xi{A'}+\xi{P_{\xi}}|}{|\xi{P_{\xi}}|}\ll{\frac{|A+A|}{N}}\leq{\frac{K|A|}{N}},$$
translates of $\xi{P_{\xi}}$, which is a subset of $A$. Similarly,
90\% of $-\xi{A'}$ can be covered by `at most'
$$\frac{|-\xi{A'}-\xi{P_{\xi}}|}{|\xi{P_{\xi}}|}\ll{\frac{|A+A|}{N}}\leq{\frac{K|A|}{N}},$$
translates of $\xi{P_{\xi}}\subset{A}$.

\end{proof}

\subsection{Five Cases}

For a given set $B\subset{F}$, recall from \eqref{Rdefn} the definition of $R(A)$. The remainder of the proof is now divided into five cases corresponding to the nature of the sets $R(\tilde{A}_{x_0})$ and $R(B_{y_0})$.

\textbf{Case 1}: Suppose $R(\tilde{A}_{x_0})\neq{R(B_{y_0})}$.

Case 1.1: There is some element $r\in{R(\tilde{A}_{x_0})}$ such that $r\notin{R(B_{y_0})}$. Fix this $r=\frac{a_1-a_2}{a_3-a_4}$
 and elements $a_1,a_2,a_3,a_4\in{\tilde{A}_{x_0}}$ representing it. Since $r\notin{R(B_{y_0})}$, there exist only trivial solutions to
\begin{equation}
b_1+rb_2=b_3+rb_4,
\label{cseq}
\end{equation}
such that $b_1,b_2,b_3,b_4\in{B_{y_0}}$. Let $B_{y_0}'$
 be some subset of $B_{y_0}$, with cardinality $|B_{y_0}'|\approx{|B_{y_0}|}$.
 This subset is required for the benefit of applying the covering lemma, and can be specified later.
 The absence of non-trivial solutions to \eqref{cseq} implies that
$$|B_{y_0}'|^2={|B_{y_0}'+rB_{y_0}'|}.$$
So,
$$\left(\frac{LN}{|A|}\right)^2\ll{|a_1B_{y_0}'-a_2B_{y_0}'+a_3B_{y_0}'-a_4B_{y_0}'|}.$$
By Application \ref{theorem:application}, each of
$a_1B_{y_0},-a_2B_{y_0},a_3B_{y_0}$ and $-a_4B_{y_0}$ can be 90\%
covered by $\ll{\frac{K|A|}{N}}$ translates of $A$. By choosing an
appropriate subset $B_{y_0}'$, we can ensure that each of the four
terms in the above sumset get fully covered. Therefore, Application
\ref{theorem:application} is used four times in order to deduce that
$$\left(\frac{LN}{|A|}\right)^2\ll{\left(\frac{K|A|}{N}\right)^4|A+A+A+A|}\ll{\left(\frac{K|A|}{N}\right)^4K^3|A|}.$$
The above inequality can be rearranged into the form
$$M^2N^2\ll{K^7|A|^7}.$$
An application of \eqref{Nbound} and then subsequently
\eqref{dyadic} implies that
$$K\gg{\frac{|A|^{1/11}}{(\log_2|A|)^{4/11}}},$$
which is a better bound than required.

Case 1.2: There is some $r\in{R(B_{y_0})}$ such that
$r\notin{R(\tilde{A}_{x_0})}$. Fix this $r=\frac{p-q}{s-t}$ as well
as elements $p,q,s,t\in{B_{y_0}}$ representing $r$.  By the
definition of $B_{y_0}$, each of $\frac{y_0}{p}, \frac{y_0}{q},
\frac{y_0}{s}$ and $\frac{y_0}{t}$ belongs to $\Xi$, the set of
slopes supporting $\approx{N}$ points from $P$. It can be observed
that $P$ is symmetric through the line $y=x$, thus
$\frac{p}{y_0},\frac{q}{y_0},\frac{s}{y_0}$ and $\frac{t}{y_0}$ are
also elements of $\Xi$.

Let $\tilde{A}_{x_0}'$ be a positively proportioned subset of
$\tilde{A}_{x_0}$, to be chosen later in order to apply the covering
lemma. Now, since $r\notin{R(\tilde{A}_{x_0}')}$, there exist only
trivial solutions to the equation
$$a_1+ra_2=a_3+ra_4,$$
such that $a_1,a_2,a_3,a_4\in{\tilde{A}_{x_0}'}$. Therefore,
\begin{align*}
|\tilde{A}_{x_0}|^2\approx{|\tilde{A}_{x_0}'|^2}&={|\tilde{A}_{x_0}'+r\tilde{A}_{x_0}'|}
\\&\leq{\left|\frac{p}{y_0}\tilde{A}_{x_0}'-\frac{q}{y_0}\tilde{A}_{x_0}'+\frac{s}{y_0}\tilde{A}_{x_0}'-\frac{t}{y_0}\tilde{A}_{x_0}'\right|}.
\end{align*}
By Application \ref{theorem:application},
each of $\frac{p}{y_0}\tilde{A}_{x_0},-\frac{q}{y_0}\tilde{A}_{x_0},\frac{s}{y_0}\tilde{A}_{x_0}$
and $-\frac{t}{y_0}\tilde{A}_{x_0}$ can be 90\% covered by $\ll{\frac{K|A|}{N}}$ translates of $A$.
The subset $\tilde{A}_{x_0}'$ may be chosen earlier so that
each of $\frac{p}{y_0}\tilde{A}_{x_0}',-\frac{q}{y_0}\tilde{A}_{x_0}',\frac{s}{y_0}\tilde{A}_{x_0}'$
and $-\frac{t}{y_0}\tilde{A}_{x_0}'$ are covered in their entirety by the translates of $A$.

Therefore, four applications of the covering lemma, along with
\eqref{mishalemma1} and \eqref{finaltouch} yield
$$\left(\frac{LM}{|A|^3}\right)^2\ll{\left(\frac{K|A|}{N}\right)^4|A+A+A+A|}\ll{\frac{K^7|A|^5}{N^4}}.$$
This can be rearranged to give
$$M^4\ll{K^7|A|^{11}}.$$
Finally, apply \eqref{dyadic} to deduce that
$$K\gg{\frac{|A|^{1/11}}{(\log_2|A|)^{4/11}}}.$$

From this point forward, we may assume that $R(\tilde{A}_{x_0})=R(B_{y_0})$.

\textbf{Case 2}: Suppose
$1+R(\tilde{A}_{x_0})\not\subset{R(\tilde{A}_{x_0})}=R(B_{y_0})$.

So, there exist elements $p,q,s,t\in{\tilde{A}_{x_0}}$ such that
$$r:=1+\frac{p-q}{s-t}\notin{R(\tilde{A}_{x_0})}=R(B_{y_0}).$$
Let $\tilde{A}_{\tilde{p}}$ be as defined in Lemma
\ref{theorem:misha}. Identify two subsets of positive proportion,
$\tilde{A}_{\tilde{p}}'\subset{\tilde{A}_{\tilde{p}}}$ and
$B_{y_0}'\subset{B_{y_0}}$, to be specified later for the purpose of
applying the covering lemma. Applying Lemma \ref{theorem:CPlun} with
$X=B_{y_0}'$, there exists a further subset
$B_{y_0}''\subseteq{B_{y_0}'}$, with
$|B_{y_0}''|\approx{|B_{y_0}'|}$, such that
\begin{equation}
\left|B_{y_0}''+\tilde{A}_{\tilde{p}}'+\left(\frac{p-q}{s-t}\right)\tilde{A}_{\tilde{p}}'\right|\ll{\frac{|A+A|}{|B_{y_0}|}\left|B_{y_0}'+\left(\frac{p-q}{s-t}\right)\tilde{A}_{\tilde{p}}'\right|}
\label{plunnecke}
\end{equation}
Now, since $B_{y_0}''$ and $\tilde{A}_{\tilde{p}}'$ are subsets of
$B_{y_0}$, there exist only trivial solutions to
$$x_1+rx_2=x_3+rx_4,$$
such that $x_1,x_3\in{B_{y_0}''}$ and
$x_2,x_4\in{\tilde{A}_{\tilde{p}}'}$, as otherwise
$r\in{R(B_{y_0})}$. This implies that
$$\left(\frac{LN}{|A|}\right)\left(\frac{LMN}{|A|^4}\right)\ll{|B_{y_0}''||\tilde{A}_{\tilde{p}}'|}={|B_{y_0}''+r\tilde{A}_{\tilde{p}}'|},$$
where the leftmost inequality is a consequence of
\eqref{mishalemma2} and the lower bound on $|B_{y_0}|$ established
earlier in Lemma \ref{theorem:misha}. Combining this inequality with
\eqref{plunnecke} and rearranging gives
\begin{equation}
\left(\frac{LN}{|A|}\right)^2\left(\frac{LMN}{|A|^4}\right)\ll{K|A|\left|B_{y_0}'+\left(\frac{p-q}{s-t}\right)\tilde{A}_{\tilde{p}}'\right|}.
\label{messy}
\end{equation}
Clearly,
$$\left|B_{y_0}'+\left(\frac{p-q}{s-t}\right)\tilde{A}_{\tilde{p}}'\right|\leq{|sB_{y_0}'-tB_{y_0}'+p\tilde{A}_{\tilde{p}}'-q\tilde{A}_{\tilde{p}}'|}.$$
Note also that
$p\tilde{A}_{\tilde{p}}\subseteq{pP_{p}}\subseteq{A}$. Therefore,
$$\left|B_{y_0}'+\left(\frac{p-q}{s-t}\right)\tilde{A}_{\tilde{p}}'\right|\ll{|sB_{y_0}'-tB_{y_0}'+A-q\tilde{A}_{\tilde{p}}'|}.$$
Finally, three applications of the covering lemma are required. By
Application \ref{theorem:application}, $sB_{y_0}$ and $-tB_{y_0}$
can be 90\% covered by `at most' $\frac{K|A|}{N}$ translates of $A$.
The earlier choice of $B_{y_0}'$ should be made so as to ensure that
both $sB_{y_0}'$ and $-tB_{y_0}'$ get fully covered by these
translates of $A$. Similarly, $-q\tilde{A}_{\tilde{p}}'$ can be
fully covered by $\ll{\frac{K|A|}{N}}$ translates of $A$. It follows
that
$$\left|B_{y_0}'+\left(\frac{p-q}{s-t}\right)\tilde{A}_{\tilde{p}}'\right|\ll{|A+A+A+A|\left(\frac{K|A|}{N}\right)^3}\ll{\frac{K^6|A|^4}{N^3}}.$$
Combining the above with \eqref{messy} and rearranging gives
$$M^4\ll{K^7|A|^{11}}.$$
Finally, an application of \eqref{dyadic}, leads to the conclusion
that
$$K\gg{\frac{|A|^{1/11}}{(\log_2|A|)^{4/11}}}.$$

\textbf{Case 3}: Suppose $\tilde{A}_{x_0}\nsubseteq{R(\tilde{A}_{x_0})}=R(B_{y_0})$. 

Then we can find an element
$z\in{\tilde{A}_{x_0}}$ such that $z\notin{R(B_{y_0})}$. Recall from
its definition in \eqref{mishalemma2}, the set
$\tilde{A}_{\tilde{z}}$, which is a subset of $B_{y_0}$. Since
$z\notin{R(B_{y_0})}$, there exist only trivial solutions to
$$a_1+za_2=a_3+za_4,$$
such that $a_1,a_3\in{B_{y_0}}$ and
$a_2,a_4\in{\tilde{A}_{\tilde{z}}}$. Therefore,
$$|B_{y_0}||\tilde{A}_{\tilde{z}}|={|B_{y_0}+z\tilde{A}_{\tilde{z}}|}\leq{|A+A|},$$
where the rightmost inequality uses the fact that
$z\tilde{A}_{\tilde{z}}\subset{A}$. Applying cardinality bounds
established earlier, we deduce that
$$\left(\frac{LN}{|A|}\right)\left(\frac{LMN}{|A|^4}\right)\ll{K|A|},$$
and so
$$K\gg{\frac{|A|^{1/4}}{(\log_2|A|)^{3/4}}}.$$

\textbf{Case 4}: Suppose
$\tilde{A}_{x_0}R(\tilde{A}_{x_0})\nsubseteq{R(\tilde{A}_{x_0})}=R(B_{y_0})$.

So, there must exist some $a,b,c,d,e\in{\tilde{A}_{x_0}}$ such that
$$r=a\frac{b-c}{d-e}\notin{R(B_{y_0})}.$$
Let $Y_1$ be a subset of $B_{y_0}$ to be determined later, and
recall the set $\tilde{A}_{\tilde{a}}\subset{B_{y_0}}$. Since
$r\notin{R(B_{y_0})}$, there exist only trivial solutions to
$$a_1+ra_2=a_3+ra_4,$$
such that $a_1,a_3\in{Y_1}$ and $a_2,a_4\in{\tilde{A}_{\tilde{a}}}$.
This implies that
$$|Y_1+r\tilde{A}_{\tilde{a}}|={|Y_1||\tilde{A}_{\tilde{a}}|}.$$
Let $Y_2$ be some subset of $P_b$ to be specified later. Lemma
\ref{theorem:SPlun} can be applied with $X=\frac{b-c}{d-e}Y_2$ to
bound the left hand side of the above inequality. Consequently,
\begin{align*}
|Y_1||\tilde{A}_{\tilde{a}}||Y_2|&\ll{\left|Y_1+\left(\frac{b-c}{d-e}\right)Y_2\right||a\tilde{A}_{\tilde{a}}+Y_2|}
\\&\leq{|dY_1-eY_1+bY_2-cY_2||A+A|}.
\end{align*}
The last step in the above uses the fact that
$a\tilde{A}_{\tilde{a}}\subset{A}$. Next we must apply the covering
lemma. By Application \ref{theorem:application}, $-cP_b$ can be 90\%
covered by $\ll{\frac{K|A|}{N}}$ translates of $A$.
$Y_2\subset{P_b}$ can be chosen earlier so that $-cY_2$ gets fully
covered by these translates and $|Y_2|\approx{|P_b|}\approx{N}$.
Similarly, $-e\tilde{A}_{\tilde{d}}$ can be 90\% covered by
$\ll{\frac{K|A|}{N}}$ translates of $A$. $Y_1$ can be chosen to be a
subset of $\tilde{A}_{\tilde{d}}$ such that $-eY_1$ gets fully
covered by these translates and
$|Y_1|\approx{|\tilde{A}_{\tilde{d}}|}$. After applying the covering
lemma twice, it follows that
$$|\tilde{A}_{\tilde{d}}||\tilde{A}_{\tilde{a}}|N\ll{|dY_1+A+bY_2+A||A+A|\left(\frac{K|A|}{N}\right)^2}.$$
Furthermore, $Y_1$ being a subset of $\tilde{A}_{\tilde{d}}$ implies
that $dY_1\subset{A}$, and so there is no need to apply the covering
lemma for this term. Similarly, $bY_2\subset{A}$ and hence
$$|\tilde{A}_{\tilde{d}}||\tilde{A}_{\tilde{a}}|N\ll{|A+A+A+A||A+A|\left(\frac{K|A|}{N}\right)^2}.$$
After applying \eqref{mishalemma2} and \eqref{finaltouch}, then
rearranging, we deduce
$$\left(\frac{LMN}{|A|^4}\right)^2N^3\ll{K^6|A|^4},$$
and thus
$$M^4N\ll{K^6|A|^{12}}.$$
Applying \eqref{Nbound} and then \eqref{dyadic}, it follows that
$$K\gg{\frac{|A|^{1/11}}{(\log_2|A|)^{5/11}}}.$$

\textbf{Case 5} - Suppose Cases $1\sim 4$ don't happen. Then in
particular we have
\begin{align}\label{338}
\tilde{A}_{x_0}&\subset R(\tilde{A}_{x_0});\\
1+R(\tilde{A}_{x_0})&\subset R(\tilde{A}_{x_0});\\
\tilde{A}_{x_0}R(\tilde{A}_{x_0})&\subset R(\tilde{A}_{x_0}).
\end{align}
Since $|\tilde{A}_{x_0}R(\tilde{A}_{x_0})|\geq|R(\tilde{A}_{x_0})|$,
\begin{align*}
\tilde{A}_{x_0}R(\tilde{A}_{x_0})=R(\tilde{A}_{x_0}).
\end{align*}
Noting $R(\tilde{A}_{x_0})$ is closed under reciprocation,
\begin{align*}
\frac{R(\tilde{A}_{x_0})}{\tilde{A}_{x_0}}=R(\tilde{A}_{x_0}).
\end{align*}
Given $a,x,y,z,w\in \tilde{A}_{x_0}$ with $z\neq w$,
\[a+\frac{x-y}{z-w}=a\cdot\left(1+\frac{1}{a}\cdot\frac{x-y}{z-w}\right)\in R(\tilde{A}_{x_0}).\]
This implies
\begin{align*}
\tilde{A}_{x_0}+R(\tilde{A}_{x_0})=R(\tilde{A}_{x_0}).
\end{align*}
Noting $R(\tilde{A}_{x_0})$ is additively symmetric
($R(\tilde{A}_{x_0})=-R(\tilde{A}_{x_0})$),
\begin{align*}
R(\tilde{A}_{x_0})-\tilde{A}_{x_0}=R(\tilde{A}_{x_0}).
\end{align*}
We also note
 \begin{align*}
 \tilde{A}_{x_0}\tilde{A}_{x_0}+R(\tilde{A}_{x_0})\subset \tilde{A}_{x_0}\left(\tilde{A}_{x_0}+\frac{R(\tilde{A}_{x_0})}{\tilde{A}_{x_0}}\right)=\tilde{A}_{x_0}(\tilde{A}_{x_0}+R(\tilde{A}_{x_0}))=\tilde{A}_{x_0}R(\tilde{A}_{x_0})=R(\tilde{A}_{x_0}).
 \end{align*}
Similarly,
 \begin{align*}
 \tilde{A}_{x_0}^{(n)}+R(\tilde{A}_{x_0})=R(\tilde{A}_{x_0}),
 \end{align*}
where $\tilde{A}_{x_0}^{(n)}$ is the $n$-fold productset of $\tilde{A}_{x_0}$. Consequently, for any polynomial of several variables with integer coefficients $P(x_1,x_2,\ldots,x_m)$,
 \begin{align*}
 P(\tilde{A}_{x_0},\tilde{A}_{x_0},\ldots,\tilde{A}_{x_0})+R(\tilde{A}_{x_0})=R(\tilde{A}_{x_0}).
 \end{align*}
 Applying Lemma \ref{theorem:polynomial} in the last section,  we have
$\mathbb{F}_{\tilde{A}_{x_0}}+R(\tilde{A}_{x_0})=R(\tilde{A}_{x_0}),$
where $\mathbb{F}_{\tilde{A}_{x_0}}$ is the subfield generated by
$\tilde{A}_{x_0}$. Since
\begin{align*}\mathbb{F}_{\tilde{A}_{x_0}}\subset
\mathbb{F}_{\tilde{A}_{x_0}}+R(\tilde{A}_{x_0})=R(\tilde{A}_{x_0})\subset \mathbb{F}_{\tilde{A}_{x_0}},
\end{align*}
we get $$R(\tilde{A}_{x_0})=\mathbb{F}_{\tilde{A}_{x_0}}.$$ By the
assumptions of Theorem \ref{theorem:main} and (\ref{338}),
\begin{align}\label{319}
|R(\tilde{A}_{x_0})|=|\mathbb{F}_{\tilde{A}_{x_0}}|\geq|\mathbb{F}_{\tilde{A}_{x_0}}\cap A|^2\geq|\mathbb{F}_{\tilde{A}_{x_0}}\cap \tilde{A}_{x_0}|^2=|\tilde{A}_{x_0}|^2.
\end{align}
By Lemma \ref{lemma 21}, there exist four elements
$z_1,z_2,z_3,z_4\in \tilde{A}_{x_0}$ such that for any subset $\tilde{A}_{x_0}'\subset \tilde{A}_{x_0}$
  with $|\tilde{A}_{x_0}'|\approx|\tilde{A}_{x_0}|$,
  $$|(z_1-z_2)\tilde{A}_{x_0}'+(z_3-z_4)\tilde{A}_{x_0}'|\approx|\tilde{A}_{x_0}|^2.$$
By Application \ref{theorem:application}, each of $z_1\tilde{A}_{x_0}$, $-z_2\tilde{A}_{x_0}$, $z_3\tilde{A}_{x_0}$ and $-z_4\tilde{A}_{x_0}$ can be 90\% covered by $\ll{\frac{K|A|}{N}}$ translates of $A$. Therefore, there exists a subset $\tilde{A}_{x_0}'\subset \tilde{A}_{x_0}$, with $|\tilde{A}_{x_0}'|\geq{0.6|\tilde{A}_{x_0}|}$, such that the sets $z_i\tilde{A}_{x_0}'$ are covered completely by the translates of $A$. Consequently,
\[|\tilde{A}_{x_0}|^2\ll|z_1\tilde{A}_{x_0}'-z_2\tilde{A}_{x_0}'+z_3\tilde{A}_{x_0}'-z_4\tilde{A}_{x_0}'|\ll\left(\frac{K|A|}{N}\right)^4\cdot|A+A+A+A|,\]
which implies
$$M^4\ll{|A|^{11}K^7},$$
and thus
$$K\gg{\frac{|A|^{1/11}}{(\log_2|A|)^{4/11}}}.$$

\begin{flushright}
\qedsymbol
\end{flushright}

\bibliographystyle{plain}
\bibliography{reviewbibliography}

\end{document}